\documentclass[10pt]{amsart}
\usepackage{amsmath}
\usepackage{amsthm}
\usepackage{amssymb}
\usepackage{amscd}
\usepackage{amsfonts}
\usepackage{amsbsy}

\newcommand{\ra}{\rightarrow}
\newcommand{\R}{\mathbb{R}}

\newcommand{\N}{\mathbb{N}}

\newcommand{\8}{\infty}
\newcommand{\la}{\lambda}
\newcommand{\ep}{\epsilon}

\newtheorem{thm}{Theorem}[section]

\newtheorem{lem}[thm]{Lemma}

\newtheorem{cor}[thm]{Corollary}

\begin{document}

\title[Volume growth and entropy for $C^1$ partially hyperbolic...]{Volume growth and entropy for $C^1$ partially hyperbolic diffeomorphisms}
\author{Radu Saghin}

\begin{abstract}
We show that the metric entropy of a $C^1$ diffeomorphism with a dominated splitting and the dominating bundle uniformly expanding is bounded from above by the integrated volume growth of the dominating (expanding) bundle plus the maximal Lyapunov exponent from the dominated bundle multiplied by its dimension. We also discuss different types of volume growth that can be associated with an expanding foliation and relationships between them, specially when there exists a closed form non-degenerate on the foliation. Some consequences for partially hyperbolic diffeomorphisms are presented.
\end{abstract}
\date{\today}
\maketitle

\section{Introduction and results}

In this paper we will prove an inequality for the measure-theoretical entropy of a $C^1$ partially hyperbolic diffeomorphism, and we will discuss about different types of volume growth that can be associated to the unstable foliation and relationships between them, in special under some additional topological information. The inequality can be viewed as a mixture between the well-known Pesin-Ruelle inequality between the metric entropy and the sum of the positive Lyapunov exponents, and the inequality between entropy and integrated volume growth obtained by Przytycki for $C^{1+r}$ diffeomorphisms, extended by Newhouse to $C^{1+r}$ maps, and eventually shown to be an equality for $C^{\8}$ maps by Kozlovski.

Let $M$ be a compact Riemannian manifold and $f$ a $C^1$ diffeomorphism on $M$. If $\mu$ is an ergodic invariant measure for $f$, then, by Oseledets Theorem, there exist $\lambda_1<\lambda_2<\dots <\lambda_s$ (called Lyapunov exponents), positive integers $d_1,d_2,\dots ,d_s$ (their multiplicities) with $d_1+d_2+\dots+d_s=\dim(M)$, and a measurable invariant splitting $TM=E_{\lambda_1}\oplus E_{\lambda_2}\oplus\dots\oplus E_{\lambda_s}$ (the Lyapunov splitting), with $\dim(E_{\lambda_i})=d_i$, such that for $\mu$-almost every $x\in M$ we have
$$
\lim_{n\ra\pm\8}\frac 1n\log\| Df^n(v)\|=\lambda_i,\ \forall v\in E_{\lambda_i}\setminus\{ 0\}.
$$

An invariant splitting of the tangent bundle $TM=E^1\oplus E^2$ is called {\it dominated} if
$$
\| Df|_{E^2(x)}\| >\| Df^{-1}|_{E^1(f(x))}\|^{-1},\ \forall x\in M.
$$
In other papers this condition is required to hold for some power of $f$, however one can always reduce to this simpler condition by changing the metric or replacing $f$ by some power of it. It is well-known that if there exists a dominated splitting, then the Lyapunov splitting will be subordinated to it, meaning that there exists some $r<s$ such that $E^1=E_{\lambda_1}\oplus\dots\oplus E_{\lambda_r}$ and $E^2=E_{\lambda_{r+1}}\oplus\dots\oplus E_{\lambda_s}$ (whenever the Lyapunov bundles are defined). We will denote the maximal Lyapunov exponent on $E^1$ by $\lambda_{E^1}^+$ (in this case it will be $\lambda_r$).

We define the {\it integrated volume growth of $E^2$} to be
$$
\overline{v}_{E^2}(f)=\limsup_{n\ra\8}\frac 1n\log\int_M\|\Lambda^{\dim(E^2)}Df^n|_{E^2}\| dLeb.
$$

Given an invariant measure $\mu$ for a map $f$, one can define the metric (or measure-theoretic) entropy of $f$ with respect to $\mu$, denoted $h_{\mu}(f)$, in several equivalent ways. We will use the Katok definition for ergodic $\mu$. Let $d_n$ be the metric on $M$ defined by $d_n(x,y)=\max_{0\leq i\leq n}d(f^i(x),f^i(y)),\ \forall x,y\in M$. We denote by $N(n,\delta,f,\mu)$ the minimal number of balls of $d_n$-radius $\delta$ covering a subset of $M$ of $\mu$-measure greater than one half. Then
$$
h_{\mu}(f)=\lim_{\delta\ra 0}\limsup_{n\ra\8}\frac 1n\log N(n,\delta,f,\mu).
$$

\begin{thm}\label{t}
Let $f$ be a $C^1$ diffeomorphism of the compact Riemannian manifold $M$, which has a dominated splitting $TM=E^{cs}\oplus E^u$, with the bundle $E^u$ uniformly expanding, and let $\mu$ be an ergodic invariant measure for $f$. Then
$$
h_{\mu}(f)\leq \overline{v}_{E^u}(f)+\dim(E^{cs})\lambda_{E^{cs}}^+.
$$
\end{thm}

We make some remarks about this result. First, we remind the Ruelle inequality for a $C^1$ map $f$ and an ergodic invariant measure $\mu$:
$$
h_{\mu}(f)\leq\sum_{\lambda_i>0}d_i\lambda_i.
$$
There is also the inequality obtained by Przytycki and Newhouse for $C^{1+\alpha}$ maps:
$$
h(f)\leq\limsup_{n\ra\8}\frac 1n\log\int_M\|\Lambda^*Df^n\| dLeb.
$$
Our result uses a Przytycki-type estimate in the direction of the unstable bundle $E^2$ and a (weaker) Ruelle-type estimate in the direction of the central-stable bundle $E^1$. In the case of $C^{1+\alpha}$ maps a stronger inequality was already obtained by Kozlovski (in fact he obtained a different inequality which can be put into this form for partially hyperbolic diffeomorphisms):
$$
h_{\mu}(f)\leq\overline{v}_{E^u}(f)+\sum_{0<\lambda_i\leq\lambda_{E^{cs}}^+}d_i\lambda_i.
$$
The results by Przytycki, Newhouse and Kozlovski make use of the $C^{1+\alpha}$ condition in Pesin theory, which we do not have in the $C^1$ case (we do not have the absolute continuity of $W^u$ either). However, as it was already remarked in other papers like \cite{ABC}, one has sufficient control of the expansion in the center-stable direction for almost every point, which together with the uniform expansion in the unstable direction helps us obtain the result.

In the next section we will give the proof of Theorem \ref{t}. In the last section we will discuss different types of volume growth than can be associated with a foliation and relationships between them. In particular, when there exists a closed form non-degenerate on the unstable foliation of a partially hyperbolic diffeomorphism, then all the different types of volume growth associated to the unstable foliation are equal, and locally constant (they are equal in fact with the spectral radius induced by the map in the corresponding cohomology group). We discuss some applications of these results to partially hyperbolic diffeomorphisms with the dimension of the center bundle equal to one or two. This paper is in the same spirit as \cite{HSX}, however here we consider the $C^1$ case, and we focus on a different type of volume growth and a different topological condition.

\section{Proof of Theorem \ref{t}}

We will denote for simplicity $\overline v_u=\overline v_{E^u}$, $\lambda =\lambda_{E^{cs}}^+$, $u=\dim(E^u)$ and $d=\dim(E^{cs})$. The main idea of the proof is similar to the one used by Przytycki, Kozlovski and others: we will bound $e^{nd(\lambda+\epsilon)}\int_{B_{d_n}(x,\delta)}\|\Lambda^uDf^n|_{E^u}\| dLeb$ from bellow independently of $n$, for a large set of points $x$.

The proof of the following lemma can be found in \cite{ABC}.

\begin{lem}\label{l}
Assume that $TM=E^{cs}\oplus E^u$ is a dominated splitting for the $C^1$ diffeomorphism $f$, $\mu$ is an ergodic invariant measure, and $\la$ is the maximal Lyapunov exponent corresponding to $E^{cs}$. Given $\epsilon>0$,  there exists $N_{\ep}>0$ such that for $\mu$-a.e. $x\in M$ we have
$$
\lim_{n\ra\8}\frac 1{nN_{\ep}}\sum_{i=0}^{n-1}\log\| Df^{N_{\ep}}(f^{iN_{\ep}}(x))|_{E^{cs}}\|<\la +\ep.
$$
\end{lem}

The idea of the proof of this lemma is that $\lim_{n\ra\8}\frac 1n\log\| Df^n|_{E^{cs}(x)}\|=\la$ for $\mu$-a.e. $x$, then $\lim_{n\ra\8}\frac 1n\int_M\log\| Df^{n}|_{E^{cs}(x)}\| d\mu=\la$, so choose $N_{\ep}$ such that $\int_M\frac 1{N_{\ep}}\log\| Df^{N_{\ep}}|_{E^{cs}(x)}\| d\mu<\la+\ep$, and apply the Birkhoff Ergodic Theorem to the diffeomorphism $f^{N_{\ep}}$ and the real valued function $x\ra\frac 1{N_{\ep}}\log\| Df^{N_{\ep}}|_{E^{cs}(x)}\|$ (one has to be a bit careful because $\mu$ may not be ergodic for $f^{N_{\ep}}$).

The bundle $E^u$ can be integrated to form the unstable foliation $W^u$, but the center-stable
bundle $E^{cs}$ may not be integrable. However inside a ball of radius 1 (eventually after
rescaling the metric) around every point $x\in M$ one can construct fake foliations (or plaques)
$\tilde W^{cs}_x$, with $C^1$ leaves and the tangent space $T\tilde W^{cs}_x$ depending
continuously on the point, uniformly with respect to $x$, such that $T_x\tilde
W^{cs}_x(x)=E^{cs}_x$, and $f(\tilde W^{cs}_{x,r_0}(x))\subset\tilde W^{cs}_{f(x)}(f(x))$ for some
$r_0$ independent of $x$ (here $\tilde W^{cs}_{x,r_0}(x)$ is the ball of radius $r_0$ in $\tilde
W^{cs}_x(x)$ with the induced metric).

There exist $k_{\ep}>0$, $A_{\ep}\subset M$, $\mu(A_{\ep})>\frac 12$, such that for every $x\in
A_{\ep}$, and for every $k\geq k_{\ep}$, we have
$$
\frac 1{kN_{\ep}}\sum_{i=0}^{k-1}\log\| Df^{N_{\ep}}(f^{iN_{\ep}}(x))|_{E^{cs}}\|<\la +\ep.
$$

There exist $0<r_{\ep}<r_0$, and a backward invariant cone field $C^{cs}_{\ep}$ around $E^{cs}$,
such that for any $x,y\in M$ with $d(x,y)<r_{\ep}$, and for any $d$-dimensional subspaces $E_x$
and $E_y$ of $T_xM$ and $T_yM$ respectively, which are tangent to $C^{cs}_{\ep}$, we have
$$
|\log\| Df^{N_{\ep}}|_{E_x}\|-\log\| Df^{N_{\ep}}|_{E_y}\| |<N_{\ep}\ep.
$$
By making eventually $r_{\ep}$ smaller, we can assume that $T\tilde W^{cs}_{x,r_{\ep}}(x)$ is
tangent to $C^{cs}_{\ep}$ for every $x\in M$.

For every $k>0$ and every $y\in\tilde W^{cs}_{x,r_{\ep}}(x)$ such that $f^i(y)\in\tilde
W^{cs}_{f^i(x),r_{\ep}}(f^i(x))$, for all $0\leq i<kN_{\ep}$, we have
$$
\| Df^{kN_{\ep}}|_{T\tilde W^{cs}_x(y)}\|\leq\prod_{i=0}^{k-1}\| Df^{N_{\ep}}|_{T\tilde
W^{cs}_{f^{iN_{\ep}}(x)}(f^{iN_{\ep}}(y))}\|<e^{kN_{\ep}\ep}\prod_{i=0}^{k-1}\|
Df^{N_{\ep}}|_{E^{cs}_{f^{iN_{\ep}}(x)}}\|.
$$

Let $L>\max_{x\in M}\| Df|_{C^{cs}_x}\|$, $L>1$, and $n_{\ep}=(k_{\ep}-1)N_{\ep}\frac{\log
L}{\la+2\ep}$.

\begin{lem}\label{l2}
For every $x\in A_{\ep}$, $r\leq r_{\ep}$, $n\geq n_{\ep}$, and for all $0\leq i\leq n$, we have
$f^i(\tilde W^{cs}_{x,r(n,r)}(x))\subset\tilde W^{cs}_{f^i(x),r/2}(f^i(x))$, where $r(n,r)=\frac 12L^{-N_{\ep}}e^{-n(\la+2\ep)}r$.
\end{lem}

\begin{proof}
We first remark that if $s<r<r_0$ then $f(\tilde W^{cs}_{x,s}(x))\subset\tilde
W^{cs}_{f(x),Ls}(f(x))$, because $L$ is an upper bound for the expansion in $\tilde W^{cs}_{x,r}$.
Using induction one can show that if $r(n,r)L^{k_{\ep}N_{\ep}}\leq \frac r2$ (*) then the first
$k_{\ep}N_{\ep}$ iterates of $\tilde W^{cs}_{x,r(n,r)}(x)$ stay inside disks of radius $r/2$ of
the fake center-stable foliation around the iterates of $x$.

If $x\in A_{\ep}$, $k\geq k_{\ep}$ and $y\in\tilde W^{cs}_{x,s}(x)$ such that $f^i(y)\in\tilde
W^{cs}_{f^i(x),r_{\ep}}(f^i(x))$, for all $0\leq i<n=kN_{\ep}+l$, $0\leq l<N_{\ep}$, we have
$$
\| Df^n|_{T\tilde W^{cs}_x(y)}\|<e^{kN_{\ep}\ep}L^l\prod_{i=0}^{k-1}\|
Df^{N_{\ep}}|_{E^{cs}_{f^{iN_{\ep}}(x)}}\|<e^{kN_{\ep}(\la+2\ep)}L^l,
$$
so
$$
f^n(\tilde W^{cs}_{x,s}(x))\subset\tilde W^{cs}_{f^n(x), se^{kN_{\ep}(\la+2\ep)}L^l}(f^n(x))\subset\tilde W^{cs}_{f^n(x),se^{n(\la+2\ep)}L^{N_\ep}}(f^n(x)).
$$

Therefore, if $r(n,r)L^{k_{\ep}N_{\ep}}\leq \frac r2$ (*) and $r(n,r)e^{n(\la+2\ep)}L^{N_{\ep}}\leq \frac r2$ (**) are satisfied, then again we can obtain by induction  that the first $n$ iterates of $\tilde W^{cs}_{x,r(n,r)}(x)$ stay inside disks of radius $r/2$ of the fake center-stable foliation around the iterates of $x$. The condition $n\geq n_{\ep}$ for our choice of $n_{\ep}$ is equivalent to the fact that inequality (**) implies inequality (*), while the inequality (**) is equivalent to $r(n,r)\leq\frac 12L^{-N_{\ep}}e^{-n(\la+2\ep)}r$, q.e.d.
\end{proof}

Now we continue with the proof of Theorem \ref{t}. Lemma \ref{l2} shows that for every $x\in
A_{\ep}$, $r<r_{\ep}$ and $n\geq n_{\ep}$, we have $\tilde W^{cs}_{x,r(n,r)}(x)\subset
B_{d_n}(x,\frac r2)\subset B_{d_n}(x,r)$, where $r(n,r)=\frac 12L^{-N_{\ep}}e^{-n(\la+2\ep)}r$.

Let $C^u$ be a forward invariant cone field around $E^u$ such that $Df$ is uniformly expanding on
$C^u$. For every $x\in M$ consider $D^u_x$ to be a smooth (uniformly with respect to $x$)
foliations of $B(x,1)$ tangent to $C^u$, which are uniformly absolutely continuous: there exist continuous real-valued functions $\alpha_x:\tilde W^{cs}_x(x)\cap B(x,1)\ra\R$ and $\beta_x:B(x,1)\ra\R$ such that for every integrable function $h:M\ra\R$ we have
$$
\int_{B(x,1)}h(z)dLeb=\int_{\tilde W^{cs}_x(x)\cap
B(x,1)}\alpha_x(y)\int_{D^u_x(y)}\beta_x(z)h(z)dLeb_{D^u_x(y)}dLeb_{\tilde W^{cs}_x(x)}.
$$
Furthermore there exists a constant $C(f)>0$ such that $\frac 1{C(f)}<\alpha_x,\beta_x<C(f)$ for all $x\in M$.

For every $x\in M$, $n>0$ and $y\in\tilde W^{cs}_{x,r}(x)$, define $D^n_x(y)$ as
the connected component of $f^n(D^u_x(y))\cap B(f^n(x),r)$ containing $f^n(y)$. One can show by induction that, for every $x\in A_{\epsilon}$, $n\geq n_{\epsilon}$, $y\in\tilde W^{cs}_{x,r(n,r)}(x)$ and any $0\leq i\leq n$, the distance between $f^i(y)$ and the boundary of $D^i_x(y)$ (measured inside $f^i(D^u_x(y))$) is greater or equal than $r/2$ (this is because $d(f^i(x),f^i(y))\leq\frac r2$). Let $D^n_{x,y}$ be the ball of radius $r/2$ inside $D^n_x(y)$. Since $D^n_x(y)$ is tangent to $C^u$ for all $i>0$, we obtain that there exists a constant $C(C^u)>0$ such that $\hbox{vol}(D^n_{x,y})>C(C^u)r^u$ for all $x\in A_{\epsilon}$, $n\geq n_{\ep}$, $r<r_{\ep}$, and all $y\in\tilde W^{cs}_{x,r(n,r)}(x)$ ($C(C^u)$ is independent of $x$ and $n$). Also $D^n_{x,y}$ and at least $n$ pre-images of it are tangent to $C^u$, so uniformly contracting under $f^{-1}$, and because $d(f^i(x),f^i(y))<\frac r2$ for all $0\leq i\leq n$, we get that $f^{-n}(D^n_{x,y})\subset B_{d_n}(x,r)$, for all $x\in A_{\epsilon}, n\geq n_{\epsilon}, r<r_{\epsilon}$ and $y\in\tilde W^{cs}_{x,r(n,r)}(x)$.

In conclusion for every $x\in A_{\ep}$, $n\geq n_{\ep}$ and $r<r_{\ep}$, we have
$$
\cup_{y\in\tilde W^{cs}_{x,r(n,r)}(x)}f^{-n}(D^n_{x,y})\subset B_{d_n}(x,r).
$$
There exist a constant $C'(f)>0$ such that $\hbox{vol}(W^{cs}_{x,r(n,r)})\geq C'(f)r(n,r)^d$. Then
\begin{eqnarray*}
I&=&\int_{B_{d_n}(x,r)}\|\Lambda^uDf^n|_{TD^u_x}\| dLeb\\
&\geq&\int_{\tilde W^{cs}_{x,r(n,r)}(x)}\alpha_x(y)\int_{f^{-n}(D^n_{x,y})}\beta(z)\|\Lambda^uDf^n(z)|_{TD^u_x}\| dLeb_{D^u_x(y)}dLeb_{\tilde W^{cs}_{x,r(n,r)}(x)}\\
&\geq&\frac 1{C(f)^2}\int_{\tilde W^{cs}_{x,r(n,r)}(x)}\int_{f^{-n}(D^n_{x,y})}\|\Lambda^uDf^n(z)|_{TD^u_x}\| dLeb_{D^u_x(y)}dLeb_{\tilde W^{cs}_{x,r(n,r)}(x)}\\
&\geq&\frac 1{C(f)^2}\int_{\tilde W^{cs}_{x,r(n,r)}(x)}\hbox{vol}(D^n_{x,y})dLeb_{\tilde W^{cs}_{x,r(n,r)}(x)}\geq\frac{C(C^u)C'(f)}{C(f)^2}r^ur(n,r)^d\\
&=&C(f,\ep,r)e^{-nd(\la+2\ep)}.
\end{eqnarray*}

We also know that $E^{cs}\oplus E^u$ is a dominated splitting and $TD^u$ is close to $E^u$, so
there exists a constant $C''(f)>0$ such that
$$
\frac
1{C''(f)}\|\Lambda^uDf^n(z)|_{TD^u_x}\|<\|\Lambda^uDf^n(z)|_{E^u}\|<C''(f)\|\Lambda^uDf^n(z)|_{TD^u_x}\|,
$$
for all $x\in M$ and $n>0$. Consequently there exists $c(f,\ep,r)>0$ such that for every $x\in A_{\ep}$, $r<r_{\ep}$ and $n\geq n_{\ep}$, we have
$$
\int_{B_{d_n}(x,r)}\|\Lambda^uDf^n(z)|_{E^u}\| dLeb\geq c(f,\ep,r)e^{-nd(\la+2\ep)}.
$$

Now let $S\subset A_{\ep}$ be a maximal $(n,r)$-separated set in $A_{\ep}$. Then $S$ is also a $(n,2r)$-spanning set for $A_{\ep}$, so the cardinality of $S$ satisfies the inequality $|S|\geq N(n,2r,f,\mu)$. Then for $r<r_{\ep}$ and $n>n_{\ep}$ we have
\begin{eqnarray*}
\int_M\|\Lambda^uDf^n(z)|_{E^u}\| dLeb&\geq &\sum_{x\in S}\int_{B_{d_n}(x,r)}\|\Lambda^uDf^n(z)|_{E^u}\| dLeb\\
&\geq &c(f,\ep,r)N(n,2r,f,\mu)e^{-nd(\la+2\ep)},
\end{eqnarray*}
so
$$
\overline{v}_{E^u}(f)\geq -d\la-2d\ep+\limsup_{n\ra\8}\frac 1n\log N(n,2r,f,\mu),
$$
and taking $r\ra0$ we get that $\overline{v}_{E^u}(f)\geq h_{\mu}(f)-d\la-2d\ep$. But this is true for every $\ep>0$, so $\overline{v}_{E^u}(f)\geq h_{\mu}(f)-d\la$, and this finishes the proof of Theorem \ref{t}.

\section{Different types of volume growth}

Besides the integrated volume growth which appears in Theorem \ref{t}, there are other types of volume growth which can be associated to an invariant foliation, or more restrictively to an unstable foliation. We will consider again $f$ to be a $C^1$ diffeomorphism of the compact Riemannian manifold $M$ with a dominated splitting $TM=E^{cs}\oplus E^u$, with $E^u$ uniformly expanding, which implies that $E^u$ is uniquely integrable to form the invariant unstable foliation $W^u$ with $C^1$ leaves. Let $\tilde d$ be the Riemannian metric induced on the leaves of $W^u$, and consider the family of disks of radius one inside the leaves of the foliation $W^u$:
$$
\mathcal F(W^u)=\{ B_{\tilde d}(x,1)\subset W^u(x),\ x\in M\}.
$$
We define the volume growth, respectively the absolute volume growth of $W^u$ under $f$ to be
\begin{eqnarray*}
v_u(f)&=&\sup_{D\in\mathcal F(W^u)}\limsup_{n\ra\8}\frac 1n \log\hbox{vol}(f^n(D));\\
v_u^a(f)&=&\limsup_{n\ra\8}\frac 1n \sup_{D\in\mathcal F(W^u)}\log\hbox{vol}(f^n(D)).
\end{eqnarray*}
These invariants are independent of the Riemannian metric, and can be defined in fact for any invariant foliation with $C^1$ leaves (not only the unstable one). The exponential rate of growth of the volume of {\it any} $C^1$ disk inside a leaf of a foliation under iterates of $f$ is bounded from above by $v_u(f)$.

We will also consider a family of disks uniformly transverse to $E^{cs}$ and of uniformly bounded size: let $u=\dim(E^u)$, $\delta<\min_{x\in M}\angle(E^u(x),E^{cs}(x))$, and define
$$
\tilde F(W^u)=\{ D=\hbox{Image}(g):g\in C^1(B(0,1)\subset\R^u, M),\| Dg\|<1, \angle(TD,E^u)<\delta\}.
$$
We define the extended volume growth, respectively the absolute extended volume growth of $W^u$ with respect to $f$ to be
\begin{eqnarray*}
\tilde v_u(f)&=&\sup_{D\in\mathcal{\tilde F}(W^u)}\limsup_{n\ra\8}\frac 1n \log\hbox{vol}(f^n(D));\\
\tilde v_u^a(f)&=&\limsup_{n\ra\8}\frac 1n \sup_{D\in\mathcal{\tilde F}(W^u)}\log\hbox{vol}(f^n(D)).
\end{eqnarray*}
Again these invariants are independent of the Riemannian metric, and can be defined in fact for any invariant splitting. The exponential rate of growth of the volume of {\it any} $C^1$ disk uniformly transverse to $E^{cs}$ under iterates of $f$ is bounded from above by $\tilde v_u(f)$. We remind also that
$$
\overline{v}_u(f)=\limsup_{n\ra\8}\frac 1n\log\int_M\|\Lambda^uDf^n|_{E^u}\| dLeb.
$$

Clearly $v_u\leq v_u^a,\tilde v_u$ and $v_u^a,\tilde v_u\leq\tilde v_u^a$. Also $\overline v_u\leq\tilde v_u^a$, and if the disintegrations of the Lebesgue measure on $M$ along the leaves of $W^u$ are absolutely continuous with respect to the Lebesgue measure on leaves with densities uniformly bounded from above (which is the case for $W^u$ for $C^{1+\alpha}$ diffeomorphisms), then $\overline v_u\leq v_u^a$. It is easy to construct examples where $\overline v_u<v_u$. We are not aware of any example when $v_u,v_u^a,\tilde v_u$ and $\tilde v_u^a$ are different, it is possible that some of the inequalities above are in fact equalities for every partially hyperbolic diffeomorphism, and one always has the inequality $\overline v_u\leq v_u$. One can prove that all the five types of volume growth are equal when there exists a smooth closed form $\omega\in\Omega^u(M)$ which is non-degenerate on $W^u$. In fact in this case the volume growth is equal to the logarithm of the spectral radius of the map induced by $f$ on $H^u(M,\R)$, and is constant for small perturbations of $f$.

\begin{thm}
Let $f$ be a $C^1$ partially hyperbolic diffeomorphism on the compact manifold $M$ such that there exists a closed $u$-form $\omega$ which is non-degenerate on the unstable foliation $W^u$ (which has dimension $u$). Then all the five types of volume growth associated with the unstable foliation which we defined above are equal with $\log\hbox{spec}(f^*_u)$, where $f^*_u:H^u(M,\R)\ra H^u(M,\R)$ is the map induced by $f$ on the $u$-cohomology group of $M$. The same is true for all $C^1$ close enough diffeomorphisms.
\end{thm}

\begin{proof}
The main idea of the proof is that the rate of growth of nearby disks in the unstable direction can be related using the closed form $\omega$. Fix two continuous invariant cone fields $C^u$ and $D^u$ around the unstable bundle $E^u$ such that $\omega$ is non-degenerate on the closure of $D^u$ and $C^u$ is strictly inside $D^u$. We say that $D$ is an $u$-disc of size $r$ centered at $x$ if $D$ is the image of a $C^1$ embedding of the unit $u$-dimensional disk into $M$, tangent to $C^u$, and, with respect to the metric induced on $D$, $d_D(x,y)=r$ for all $y\in\partial D$. We will assume that every $u$-disk $D$ is oriented such that $\int_D\omega>0$. We have the following Lemma.

\begin{lem}\label{l}
Let $f$ be a partially hyperbolic diffeomorphism on the compact manifold $M$ such that there exists a closed $u$-form $\omega$ which is non-degenerate on the unstable foliation $W^u$ (which has dimension $u$), and $C^u$ and $D^u$ be invariant cones as above. Then there exist $\delta,r,R,C>0$ suh that for any $x,y\in M$, $d(x,y)<\delta$, any $u$-disk $D_x$ of size $r$, any $u$-disk $D_y$ of size $R$, and any $n>0$, we have $\hbox{vol}(f^n(D_x))\leq C\hbox{vol}(f^n(D_y))$.
\end{lem}

\begin{proof}
We choose $\delta$, $R$ and $\frac rR$ small enough such that for any $u$-disks $D_x$ and $D_y$ of sizes $r$, respectively $R$, with $d(x,y)<\delta$, there exists a $C^1$ $u$-dimensional submanifold $S$ tangent to $D^u$, diffeomorphic with $\mathbb S^{u-1}\times [0,1]$, such that $S+D_x-D_y$ is the boundary (in the simplicial sense) of a $C^1$, $(u+1)$-dimensional submanifold ($S$ is again oriented such that $\int_S\omega>0$). There exists $C_1>0$ such that for every $u$-dimensional submanifold $D\subset M$ tangent to $C^u$ we have $\frac 1{C_1}\hbox{vol}(D)\leq|\int_D\omega|\leq C_1\hbox{vol}(D)$. If $f^n$ preserves the orientation of $D_x,D_y$ and $S$, we obtain:
\begin{eqnarray*}
\hbox{vol}(f^n(D_x))&\leq &C_1\int_{f^n(D_x)}\omega=C_1\int_{f^n(D_y)}\omega-C_1\int_{f^n(S)}\omega\\
&\leq &C_1\int_{f^n(D_y)}\omega\leq C_1^2\hbox{vol}(f^n(D_y)).
\end{eqnarray*}

We used the fact that $\omega$ is closed and non-degenerate on the invariant cone field $D^u$. If $f^n$ does not preserve the orientation of $D_x,D_y$ or $S$, then it must reverse all of them, because iterates of submanifolds tangent to $D^u$ stay in $D^u$ and $D_x\cup D_y\cup S$ is connected, and the same result follows. Let $C=C_1^2$ q.e.d.
\end{proof}

In order to prove that all the different types of volume growth are equal, it is enough to show that $\tilde v^a_u(f)\leq\overline{v}_u(f)\leq v_u(f)$.

\noindent
{\bf Claim 1} $\tilde v^a_u(f)\leq\overline{v}_u(f)$.

We will use the definitions from the previous section for the local foliations $\tilde W^{cs}_x$ and $D^u_x$. For every $x\in M$, let
$$
U_x=\cup_{y\in\tilde W^{cs}_{x,\delta}(x)}D^u_{x,y}(R),
$$
where $D^u_{x,y}(R)$ is the ball of radius $R$ centered at $y$ in $D^u_{x,y}$, the leaf of the local foliation $D^u_x$ passing through $y$ (we can assume that $R$ and $\delta$ are small enough).

There exists a sequence of disks $D_n\in\tilde F(W^u)$, such that
$$
\tilde v^a_u(f)=\limsup_{n\ra\8}\frac 1n\log\hbox{vol}(f^n(D_n)).
$$
Without loss of generality we can assume that $D_n$ is a $u$-disk of size $r$ centered at $x_n$. We obtain ($C$ represents different constants independent of $n$):
\begin{eqnarray*}
\frac 1n\log\int_M\|\Lambda^uDf^n|_{E^u}\| dLeb&>&\frac 1n\log\int_{U_{x_n}}\|\Lambda^uDf^n|_{E^u}\| dLeb\\
&\geq&\frac 1n\log\int_{U_{x_n}}C\|\Lambda^uDf^n|_{TD^u_{x_n}}\| dLeb\\
&\geq&\frac 1n\log \int_{\tilde W^{cs}_{x_n,\delta}}C\alpha_{x_n}\int_{D^u_{x_n,y}}\beta_{x_n}\|\Lambda^uDf^n|_{TD^u_{x_n}}\| dLeb\\
&\geq&\frac 1n\log \int_{\tilde W^{cs}_{x_n,\delta}}C\hbox{vol}(f^n(D^u_{x_n,y}(R)))dLeb\\
&\geq&\frac 1n\log C+\frac 1n\log\hbox{vol}(f^n(D_n)).
\end{eqnarray*}
We used the fact that $\|\Lambda^uDf^n|_{E^u}\|$ and $\|\Lambda^uDf^n|_{TD^u_{x_n}}\|$ are comparable, $\alpha_x$ and $\beta_x$ are uniformly bounded away from zero, and Lemma \ref{l} for the last inequality. By taking the limit for $n\ra\8$ we obtain $\tilde v^a_u(f)\leq\overline{v}_u(f)$.

\noindent
{\bf Claim 2}: $\overline{v}_u(f)\leq v_u(f)$.

For every $x\in M$, let
$$
V_x=\cup_{y\in\tilde W^{cs}_{x,\delta}(x)}D^u_{x,y}(r).
$$
Because $M$ is compact, there exist $x_1,x_2,\dots x_l\in M$ such that $M=\cup_{i=1}^lV_{x_i}$. Then there exists $1\leq k\leq l$ such that
\begin{eqnarray*}
\overline{v}_u(f)&=&\limsup_{n\ra\8}\frac 1n\log\int_M\|\Lambda^uDf^n|_{E^u}\| dLeb\\
&=&\limsup_{n\ra\8}\frac 1n\log\int_{V_{x_k}}\|\Lambda^uDf^n|_{E^u}\| dLeb.
\end{eqnarray*}
For simplicity we will denote $x_k=x$. Again letting $C$ be different constants independent on $n$, we obtain:
\begin{eqnarray*}
\overline{v}_u(f)&=&\limsup_{n\ra\8}\frac 1n\log\int_{V_x}\|\Lambda^uDf^n|_{E^u}\| dLeb\\
&\leq&\limsup_{n\ra\8}\frac 1n\log\int_{V_x}C\|\Lambda^uDf^n|_{TD^u_x}\| dLeb\\
&=&\limsup_{n\ra\8}\frac 1n\log \int_{\tilde W^{cs}_{x,\delta}}C\alpha_x\int_{D^u_{x,y}(r)}\beta_x\|\Lambda^uDf^n|_{TD^u_x}\| dLeb\\
&\leq&\limsup_{n\ra\8}\frac 1n\log \int_{\tilde W^{cs}_{x,\delta}}C\hbox{vol}(f^n(D^u_{x,y}(r)))dLeb\\
&\leq&\limsup_{n\ra\8}\frac 1n\log\hbox{vol}(f^n(W^u_R(x)))\leq v_u(f),
\end{eqnarray*}
where $W^u_R(x)$ is the ball or radius $R$ centered at $x$ inside $W^u(x)$. We used again Lemma \ref{l}, the fact that $\alpha_x$ and $\beta_x$ are uniformly bounded from above, and the fact that $\|\Lambda^uDf^n|_{E^u}\|$ and $\|\Lambda^uDf^n|_{TD^u_{x_n}}\|$ are comparable.

We remark that whenever Lemma \ref{l} is true for a given partially hyperbolic diffeomorphism, it follows from the above argument that all the types of volume growth associated to the unstable foliation are equal (this could be applied for uniformly hyperbolic maps for example).

\noindent
{\bf Claim 3}: $v_u(f)=\log\hbox{spec}(f^*_u)$.

We could prove this claim using the results from \cite{Sa}, however we prefer to give a different proof, avoiding the use of currents. In order to relate the volume growth of disks with closed differential forms we need the following lemma, which says that for almost every u-disk (in some sense), the volume of the iterates of the disk grows faster than the volume of the iterates of its boundary.

\begin{lem}\label{d}
Let $f$ be a $C^1$ diffeomorphism of the compact Riemannian manifold $M$, and $\alpha:D(0,1)\subset\R^u\ra M$ a $C^1$ embedding (its image is a $C^1$ disk of dimension $u$). Assume that $f$ is uniformly expanding on $D_1=\alpha(D(0,1))$, and denote $D_r=\alpha(D(0,r))$ for $0<r\leq 1$. Then for almost every $r\in(0,1]$, we have
$$
\lim_{n\ra\8}\frac{\hbox{vol}(f^n(\partial D_r))}{\hbox{vol}(f^n(D_r))}=0,
$$
and the convergence is exponential.
\end{lem}

\begin{proof}
The uniform expansion of $f$ on $D_1$ means that there exist $c>0$, $\lambda>1$ such that
$$
\| Df^n(v)\|\geq c\lambda^n\| v\|, \ \forall n\in\N,\ \forall v\in TD_1.
$$
Let $F_n(r)=\hbox{vol}(f^n(D_r))$. Then
\begin{eqnarray*}
F_n(r)&=&\int_{D_r}\|\Lambda^uDf^n|_{TD_r}\| dLeb=\int_{D(0,r)}\|\Lambda^uDf^n(\alpha(x))|_{TD_r}\|\cdot |D\alpha|dx\\
&=&\int_0^r\int_{S(0,t)}\|\Lambda^uDf^n(\alpha(y))|_{TD_r}\|\cdot |D\alpha(y)|dydt,
\end{eqnarray*}
so we have
$$F_n\rq{}(r)=\int_{S(0,r)}\|\Lambda^uDf^n(\alpha(y))|_{TD_r}\|\cdot |D\alpha(y)|dy.
$$
On the other hand, using the fact that $\|\Lambda^uDf^n(\alpha(y))|_{TD_r}\|\geq
c\lambda^n\|\Lambda^{u-1}Df^n(\alpha(y))|_{TD_r}\|$, for some fixed constant $c>0$ and every $y\in
D(0,1)$ and $n>0$, we obtain
\begin{eqnarray*}
\hbox{vol}(f^n(\partial D_r))&=&\int_{\partial D_r}\|\Lambda^{u-1}Df^n|_{T\partial D_r}\| dLeb\\
&=&\int_{S(0,r)}\|\Lambda^{u-1}Df^n(\alpha(y))|_{T\partial D_r}\|\cdot |D\alpha(y)|_{S(0,r)}|dy\\
&\leq&\frac C{\lambda^n}F_n\rq{}(r),
\end{eqnarray*}
for some $C$ which depends only on $f$ and $\alpha$. If $\mu$ is an upper bound for
$\|\Lambda^uDf\|$ on $M$, we get that $F_n(1)\leq C\rq{}\mu^n$ for some $C\rq{}>0$, or $\log
F_n(1)\leq n\log\mu +\log C\rq{}$.

Let $1<\lambda\rq{}<\lambda$ and $r_n\in(0,1)$ such that $F_n(r_n)=1$ (clearly $r_n$ is decreasing to zero). Let $R_n=\{ t\in[r_n,1]:\ (\log F_n(t))\rq{}>\lambda\rq{}^n\}$. Because $\log F_n$ is increasing we get that $Leb(R_n)\leq\frac{n\log\mu +\log C\rq}{\lambda\rq{}^n}$. If we denote by $S_n=\cup_{i\geq n}R_i$, it is easy to see that $Leb(S_n)$ converges to zero (this is because the series $\sum_{n=1}^{\8}\frac{n\log\mu +\log C\rq}{\lambda\rq{}^n}$ is convergent). Let $S=\cap_{n=1}^{\8}S_n$, so $Leb(S)=0$.

For every $r\in(0,1]\setminus S$, there exists $n_0\in\N$ such that $r>r_{n_0}$ and $r\notin S_{n_0}$, or $(\log F_n(r))\rq{}\leq\lambda\rq{}^n$, for all $n>n_0$. But this implies that $\frac{F_n\rq{}(r)}{F_n(r)}\leq\lambda\rq{}^n$, which together with the inequality $\hbox{vol}(f^n(\partial D_r))\leq\frac C{\lambda^n}F_n\rq{}(r)$ gives that $\frac{\hbox{vol}(f^n(\partial D_r))}{\hbox{vol}(f^n(D_r))}\leq C\left(\frac{\lambda\rq{}}{\lambda}\right)^n$ for all $n>n_0$, which finishes the proof.
\end{proof}

We continue with the proof of the Claim 3. Let $\omega_1=\omega,\omega_2,\omega_3,\dots\omega_k$ be closed forms such that the cohomology classes $[\omega_1],[\omega_2],\dots[\omega_k]$ form a basis in $H^u(M,\R)$. Let $A$ be the matrix corresponding to the linear map $f_u^*:H^u(M,\R)\ra H^u(M,\R)$, and denote $A^l=(a_{ij}^l)_{1\leq i,j\leq k}$. Then $[f^{*l}(\omega)]=\sum_{i=1}^ka_{i1}^l[\omega_i]$, so $f^{*l}(\omega)-\sum_{i=1}^ka_{i1}^l\omega_i=d\omega^l$ for some $(u-1)$-form $\omega^l$. If $D$ is a u-disk then
\begin{eqnarray*}
\hbox{vol}(f^{n+l}(D))&\leq&C\left|\int_{f^{n+l}(D)}\omega\right|=C\left|\int_{f^n(D)}f^{*l}(\omega)\right|\\
&\leq&C\left|\sum_{i=1}^ka_{i1}^l\int_{f^n(D)}\omega_i\right|+C\left|\int_{\partial f^n(D)}\omega^l\right|\\
&\leq&C\rq{}\hbox{vol}(f^n(D))\max_{1\leq i\leq k}|a_{i1}^l|+C_l\hbox{vol}(\partial f^n(D)),
\end{eqnarray*}
where $C\rq{}$ is independent of $n,l$ and $C_l$ is independent of $n$. If we have that
$\lim_{n\ra\8}\frac{\hbox{vol}(f^n(\partial D_r))}{\hbox{vol}(f^n(D_r))}=0$ (which is true for
most of the disks because the Lemma \ref{d}), then
$$
\max_{1\leq i\leq k}|a_{i1}^l|\geq\frac 1{C\rq}\limsup_{n\ra\8}\frac{\hbox{vol}(f^{n+l}(D))}{\hbox{vol}(f^n(D))}
$$
From here one can easily obtain that
$$
\frac 1l\log(\max_{1\leq i\leq k}|a_{i1}^l|)+\frac{\log C\rq{}}l\geq\limsup_{n\ra\8}\frac 1n\log(\hbox{vol}(f^n(D))),
$$
and then
$$
\log\hbox{spec}(f^*_u)\geq\limsup_{l\ra\8}\frac 1l\log(\max_{1\leq i\leq k}|a_{i1}^l|)\geq\limsup_{n\ra\8}\frac 1n\log(\hbox{vol}(f^n(D))).
$$
But now every $u$-disk can be enlarged to a disk that satisfies the conclusion of Lemma \ref{d}, so $\log\hbox{spec}(f^*_u)\geq\tilde{v}_u(f)$. The opposite inequality, $\log\hbox{spec}(f^*_u)\leq\tilde{v}_u(f)$, is also true for every partially hyperbolic diffeomorphism (for a proof see \cite{SX} for example), so indeed we get that $\log\hbox{spec}(f^*_u)=v_u(f)$.

The fact that the result holds also for nearby diffeomorphisms follows from the fact that the conditions from the hypothesis are $C^1$ open.
\end{proof}

It is not difficult to show that all the types of volume growth defined above are a lower bound for the topological entropy of $f$ (it is enough to prove it for $\tilde v_u^a$, see \cite{SX} for example). As a consequence of the results from this paper, and using basically the same methods from \cite{HSX}, we get the following corollaries.

\begin{cor}
If $f$ is a $C^1$ partially hyperbolic diffeomorphism with the dimension of the center bundle equal to one, then
$$
h(f)=\max\{\overline v_u(f),\overline v_u(f^{-1})\}.
$$
\end{cor}

\begin{cor}
If $f$ is a $C^1$ partially hyperbolic diffeomorphism with the dimension of the center bundle equal to one, and there are closed $u$- and $s$-forms non-degenerated on the unstable, respectively the stable foliation, then the topological entropy is locally constant at $f$ in the space of $C^1$ diffeomorphisms, and is in fact equal to $\max\{\log\hbox{spec}f^*_u, \log\hbox{spec}f^*_{u+1}\}$.
\end{cor}

\begin{cor}
If $f$ is a $C^1$ partially hyperbolic diffeomorphism with the dimension of the center bundle equal to two, and $\mu$ is an ergodic invariant measure for $f$ such that $h_{\mu}(f)>\max\{\overline v_u(f),\overline v_u(f^{-1})\}$, then $\mu$ is hyperbolic.
\end{cor}

\begin{cor}
If $f$ is a $C^{1+\alpha}$ (or $C^{\8}$) partially hyperbolic diffeomorphism with the dimension of the center bundle equal to two, and there are closed $u$- and $s$-forms non-degenerated on the unstable, respectively the stable foliation, then the topological entropy is lower semicontinuous (or continuous) at $f$ in the $C^1$ (or $C^{\8}$) topology.
\end{cor}

{\bf Acknowledgments:} The author was partially supported by Marie Curie grant IEF-234559, and would like to thank CRM Barcelona for the hospitality, and Zhihong Xia, Jiagang Yang and Graham Smith for useful conversations.

\bibliographystyle{plain}

\end{document}